\documentclass{amsart}
\usepackage{amssymb}

\newtheorem{thm}{Theorem}[section]
\newtheorem{lem}[thm]{Lemma}

\theoremstyle{definition}

\theoremstyle{remark}

\def\underset#1\to#2{\mathop{#2}\limits_{#1}{ }}
\def\overset#1\to#2{\mathop{#2}\limits^{#1}{ }}

\numberwithin{equation}{section}

\begin{document}

\title[Geometry of the space of polynomials]{Notes on the geometry of space of polynomials}

\author{Han Ju Lee}
\address{Department of Mathematics, POSTECH, San 31, Hyoja-dong,
Nam-gu, Pohang-shi, Kyungbuk, Republic of Korea, +82-054-279-2712
} \email{hahnju@postech.ac.kr}


\thanks{The author acknowledges the financial support of the Korean
Research Foundation made in the program year of 2002
(KRF-2002-070-C00005)}


\begin{abstract}
We show that the symmetric injective tensor product space
$\hat{\otimes}_{n,s,\varepsilon}E$ is not complex strictly convex if
$E$ is a complex Banach space of $\dim E \ge 2$ and if $n\ge 2$
holds. It is also reproved that $\ell_\infty$ is finitely
represented in $\hat{\otimes}_{n,s,\varepsilon}E$ if $E$ is infinite
dimensional and if $n\ge 2$ holds, which was proved in the other way
in \cite{D}.
\end{abstract}
\maketitle                   





The real geometric properties of spaces of polynomials are discussed
in \cite{BR, RT}. In particular, it is shown that the symmetric
injective tensor product space $\hat{\otimes}_{n,s,\varepsilon}E$ is
not strictly convex if $E$ is a Banach space of $\dim E \ge 2$ and
if $n\ge 2$ holds.

Let $E$ be a Banach space over a real or complex filed and
$E^\prime$ is denoted as the Banach dual of $E$. An element $x$ in
the unit sphere $S_E$ is called a {\it (real) extreme point} of the
unit ball $B_E$ if for some $y\in E$,
\[\|x \pm y\|\le 1\] implies $y=0$. Recall that a Banach space $E$
is said to be {\it strictly convex} if every element of $S_X$ is an
extreme point of $B_E$. Suppose for the moment that $E$ is a complex
Banach space. An element $x$ in the unit sphere $S_E$ is said to be
a {\it complex extreme point} if for some $y\in E$,
\[ \sup\{ \| x+ \zeta y \| \ :\   \zeta\in \mathbf{C},\  |\zeta|=1\}
\le 1\] implies $y=0$. A complex Banach space $E$ is said to be {\it
complex strictly convex} if every point in $S_E$ is a complex
extreme point of $B_E$. Notice that if a complex Banach space is not
complex strictly convex, then it is not strictly convex.

Given a Banach space $E$, the space $\otimes_{n,s} E$ consists of
all tensors of the form \[u=\sum_{j=1}^k \lambda_j x_j\otimes
x_j\otimes x_j\cdots \otimes x_j,\] where $x_j$ are elements in $E$
for $j=1,\ldots, k$ and $\lambda_j$'s are scalars. Now define its
injective norm as
\[ \|u\|_\varepsilon = \sup\left\{
\left| \sum_{j=1}^k\  \lambda_j  \langle \phi, x_j \rangle^n \right|
\ :\ \phi \in B_{E^\prime} \right\}.\] We denote the completion of
$\otimes_{n,s}E$ with this norm by
$\hat{\otimes}_{n,s,\varepsilon}E$. It is clear that
$\hat{\otimes}_{n,s,\varepsilon}E^\prime$ is a closed subspace of
space $\mathcal{P}({}^n E)$ of all $n$-homogeneous polynomials (For
more details, consult \cite{C, D2}.)

\section{Main Results}

We begin with the following useful observation which is a
modification of Theorem~2.2 in \cite{KL}.
\begin{lem}\label{lemfunda1}
Let $F$ be a finite $m$-dimensional Banach space with the
Banach-Mazur distance $d(F,\ell_2^m)\le d$. Then there exist
 $\{x_j\}_{j=1}^m$ in $F$ and $\{\phi_j\}_{j=1}^m$ in $S_{F^\prime}$
such that for each $i=1,\ldots, m$,
\begin{equation}\label{eq0} 1 \le \langle
\phi_i, x_i \rangle \le d,\ \  \langle \phi_j, x_k \rangle=0\ \
\text{ for all different}\ \  j,k,\end{equation}
\begin{equation}\label{eqxs1}  1\le \min_{1\le k \le m} \|x_k\|
\le \max_{1\le k \le m} \|x_k\| =\|x_1\|\le d\end{equation} and for
every real $r\ge 2$,
\begin{equation}\label{eqinter2} \sup\left\{ \left(\sum_{k=1}^m |\langle \phi, x_k
\rangle|^r\right)^{\frac 1r} : \phi\in B_{F^\prime} \right\} =
\|x_1\|.\end{equation}
\end{lem}

\begin{proof}We shall identify $(\ell_2^m)'$ with $\ell_2^m$.  Consider an
isomorphism $T:\ell_2^m \rightarrow F$ with $\|T\|\le d$ and
$\|T^{-1}\|=1$. Since $\|(T^*)^{-1}\|=1$, we get $\|\phi\|\le
\|T^*(\phi)\|_2$ for every $\phi\in F'$. Then there exists $\phi_1
\in S_{F^\prime}$ such that
\[1\le \|T^*(\phi_1)\|_2 = \max\{ \|T^*(\phi)\|_2 : \phi \in B_{F^\prime}
\} \le d.\] Then there exists $\phi_2\in S_{F^\prime}$ such that
\begin{align*}1&\le \|T^*(\phi_2)\|_2\\& = \max\left\{ \|T^*(\phi)\|_2 : \phi \in
B_{F^\prime}, \ \ \langle T^*(\phi), \overline{T^*(\phi_1)} \rangle
= \langle T^*(\phi), T^*(\phi_1) \rangle_2 =0 \
\right\},\end{align*} where $\langle \cdot, \cdot \rangle_2$ is the
standard inner product in $\ell_2^m$ and $\overline{x}$ denotes
$\{\overline{x(k)}\}_{k=1}^m$ if $x= \{x(k)\}_{k=1}^m$.

In this way, we obtain $m$ vectors $\{\phi_k\}_{k=1}^m$ in
$S_{F^\prime}$ such that $\{T^*(\phi_k)\}_{k=1}^m$ are orthogonal.
Taking $w_k = \overline{T^*(\phi_k)}/\|T^*(\phi_k)\|_2$ for each
$k=1,\ldots, m$, let $x_k = T(w_k)$ for each $k=1,\ldots, m$. Notice
that
\[ \langle \phi_i, x_j \rangle = \langle \phi_i, T(w_j) \rangle =
\langle T^*(\phi_i), w_j \rangle = \|T^*(\phi_j)\|_2\delta_{ij},\]
where $\delta_{ij} = 1$ if $i=j$ and $\delta_{ij}=0$ if $i\neq j$.
 So (1) is satisfied.

Since $\{w_k\}_{k=1}^m$ is an orthonormal basis in $\ell_2^m$, we
get for each $\phi \in B_{F^\prime}$
\begin{align*}\left(\sum_{k=1}^m |\langle \phi, x_k \rangle|^2
\right)^{\frac12} &= \left(\sum_{k=1}^m |\langle \phi, T(w_k)
\rangle|^2 \right)^{\frac12} \\&=  \left(\sum_{k=1}^m |\langle
T^*(\phi), w_k \rangle|^2 \right)^{\frac12} = \|T^*(\phi)\|_2 \le
\|T^*(\phi_1)\|_2.
\end{align*}
Hence we have
\[ \sup\left\{ \left(\sum_{k=1}^m |\langle \phi, x_k \rangle |^r\right)^{\frac 1r}  : \phi\in
B_{F^\prime} \right\} = \|T^*(\phi_1)\|_2 =\langle \phi_1, x_1
\rangle \le \|x_1\|.\] It is also clear that
\[ \sup\left\{ \left(\sum_{k=1}^m |\langle \phi, x_k \rangle |^r\right)^{\frac 1r}  : \phi\in
B_{F^\prime} \right\} \ge \max_{1\le j\le m} \|x_j\|.\] Notice also
that $ 1\le \|x_j\| = \|T(w_j)\| \le d$ for every $1\le j\le m$.
Hence we have obtained $m$ vectors $\{x_1, \ldots, x_m\}$ satisfying
equations (\ref{eqxs1}) and (\ref{eqinter2}).
\end{proof}

Boyd and Ryan showed in \cite{BR} that
 $\hat{\otimes}_{n,s,\varepsilon}E$ is not real strictly convex
for every $n\ge 2$ and $\dim E\ge 2$. In the case of a complex
Banach space we have the following

\begin{thm}
Let $E$ be a real (resp. complex) Banach space of $\dim \ge 2$. Then
for $n\ge 2$ the space $\hat{\otimes}_{n,s,\varepsilon}E$ and
$\mathcal{P}({}^n E)$ are not real (resp. complex) strictly convex.
\end{thm}

\begin{proof}
We shall prove the complex case. The proof can be easily applied to
the real case. It is enough to show that
$\hat{\otimes}_{n,s,\varepsilon}E$ is not complex strictly convex.
For each two dimensional subspace $F$ of $E$, by applying
Lemma~\ref{lemfunda1}, we have two vectors $x_1, x_2$ such that
$1\le \|x_2\| \le \|x_1\|$ and
\[ \sup_{\phi\in B_{F^\prime}}  \left(|\langle \phi,
x_1 \rangle|^n+ |\langle \phi, x_2 \rangle|^n\right) = \|x_1\|^n.\]

Now let \[u =\frac{1}{\|x_1\|^n}\  x_1\otimes x_1 \otimes \cdots
\otimes x_1,\]
\[v =\frac{1}{\|x_1\|^n}\ x_2\otimes
x_2 \otimes \cdots \otimes x_2.\] Then $\|u\|_\varepsilon =1$ and
$v\neq 0$. Notice that for each $|\zeta|\le 1$,
\[\|u + \zeta v \|_\varepsilon = \frac{1}{\|x_1\|^n}
\sup_{\phi \in B_{E^\prime}} \left|\langle \phi, x_1 \rangle^n +
\zeta \langle \phi, x_2 \rangle^n \right| \le 1. \] Therefore $u$ is
not a complex extreme point. The proof is done.
\end{proof}

In the following theorem we reprove the Dineen's result in \cite{D}
by applying Lemma~\ref{lemfunda1}.

\begin{thm}\cite{D}
Let $E$ be an infinite dimensional Banach space. Then $\ell_\infty$
is finitely represented in $\hat{\otimes}_{n,s,\varepsilon}E$ and
$\mathcal{P}({}^n E)$ for every $n\ge 2$.
\end{thm}
\begin{proof}
By the Dvoretzky theorem \cite{Dv} we can choose for any positive
integer $m$ and any $\epsilon>0$ an $m$ dimensional subspace $F$ in
$E$ with $d(F, \ell_2^m) \le 1 + \epsilon$. By
Lemma~\ref{lemfunda1}, we have $m$ elements $\{x_j\}$ in $F$ and
$\{\phi_j\}$ in $S_{F^\prime}$ satisfying (\ref{eq0}), (\ref{eqxs1})
and (\ref{eqinter2}) with $d=1 + \epsilon$. For any
$(\alpha_i)_{i=1}^m\in S_{\ell_\infty^m}$, using the Hahn-Banach
extension,
\begin{align*}
\left\|\sum_{i=1}^m \alpha_i\ x_i \otimes \cdots \otimes
x_i\right\|_\varepsilon &=\sup_{\phi\in B_{E^\prime}}
\left|\sum_{i=1}^m
\alpha_i\langle \phi, x_i\rangle^n \right|\\
&\ge \max_{1\le j\le m} |\alpha_j|\langle\phi_j, x_j\rangle \ge 1,
\end{align*} and we also get
\begin{align*}
\left\|\sum_{i=1}^m \alpha_i\ x_i \otimes \cdots \otimes
x_i\right\|_\varepsilon &=\sup_{\phi\in B_{E^\prime}}
\left|\sum_{i=1}^m
\alpha_i\langle \phi, x_i\rangle^n \right|\\
&\le \sup_{\phi \in B_{E^\prime}}\sum_{i=1}^m |\alpha_i||\langle
\phi, x_i\rangle|^n \le \|x_1\|^n \le (1+\epsilon)^n.
\end{align*} The proof is complete.
\end{proof}


\begin{thebibliography}{10}

\bibitem{BR} C. Boyd and R.A. Ryan,  \emph{Geometric theory of spaces of integral
polynomials and symmetric tensor products}, J. Funct. Anal. 179
~(2001), no. 1, 18--42.

\bibitem{C}
S.B. Chae, \emph{Holomorphy and calculus in normed spaces}, Marcel
Dekker, Inc., New York, 1985.

\bibitem{D}
S. Dineen, \emph{A Dvoretzky theorem for polynomials}, Proc. Amer.
Math. Soc. 123~(1995), no. 9, 2817--2821.

\bibitem{D2} S. Dineen, \emph{Complex analysis on infinite-dimensional
spaces}, Springer-Verlag London, Ltd., London, 1999.

\bibitem{Dv} A. Dvoretzky, \emph{Some results on convex bodies and Banach
spaces}, Proc. Internat. Sympos. Linear Spaces (Jerusalem, 1960),
Jerusalem Academic Press, 1961, pp. 123--160.

\bibitem{RT} R.A. Ryan and B. Turett, \emph{Geometry of spaces of
polynomials}, J. Math. Anal. Appl. 221~(1998), no. 2, 698--711.

\bibitem{KL}
A. Kami\'nska and H.J. Lee, \emph{On uniqueness of extension of
homogeneous polynomials}, Houston J. Math. \textbf{32}~(2006), no.
1, 227--252.


\end{thebibliography}
\end{document}